\documentclass [] {article} []

\usepackage[pdftex]{graphicx}
\usepackage{amsmath}
\usepackage{amssymb}
\usepackage{amsfonts}
\usepackage{wrapfig}
\usepackage{authblk}
\usepackage{fullpage}
\usepackage{tikz}
\usetikzlibrary{snakes}
\newcommand{\aut}{\textnormal{Aut}}

\newcommand{\path}[1]{\textnormal{Path}\langle {#1} \rangle}

\newcommand{\cfpo}[1] {\mathrm{CFPO}_{#1}}
\newcommand{\alt}[1] {\mathrm{Alt}_{#1}}

\makeatletter
\providecommand*{\cupdot}{%
  \mathbin{%
    \mathpalette\@cupdot{}%
  }%
}
\newcommand*{\@cupdot}[2]{%
  \ooalign{%
    $\m@th#1\cup$\cr
    \sbox0{$#1\cup$}%
    \dimen@=\ht0 %
    \sbox0{$\m@th#1\cdot$}%
    \advance\dimen@ by -\ht0 %
    \dimen@=.5\dimen@
    \hidewidth\raise\dimen@\box0\hidewidth
  }%
}

\providecommand*{\bigcupdot}{%
  \mathop{%
    \vphantom{\bigcup}%
    \mathpalette\@bigcupdot{}%
  }%
}
\newcommand*{\@bigcupdot}[2]{%
  \ooalign{%
    $\m@th#1\bigcup$\cr
    \sbox0{$#1\bigcup$}%
    \dimen@=\ht0 %
    \advance\dimen@ by -\dp0 %
    \sbox0{\scalebox{2}{$\m@th#1\cdot$}}%
    \advance\dimen@ by -\ht0 %
    \dimen@=.5\dimen@
    \hidewidth\raise\dimen@\box0\hidewidth
  }%
}
\makeatother

\newcounter{alphabetical}
\newcounter{alphappendix}

\newenvironment{proof}{\noindent\textbf{Proof}\\}{\noindent$\Box$\\}

\newtheorem{lemma}{Lemma}[section]
\newtheorem{theorem}[lemma]{Theorem}
\newtheorem{cor}[lemma]{Corollary}

\newtheorem{dfn}[lemma]{Definition}
\newtheorem{definition}[lemma]{Definition}
\newtheorem{prop}[lemma]{Proposition}

\newtheorem{facts}[lemma]{Facts}

\title{The $\aleph_{0}$-categorical Trees and Cycle-free Partial Orders}
\author{Robert Barham \\ Institut f\"ur Algebra, TU Dresden \\ robert.barham@yahoo.co.uk}

\begin{document}

\maketitle
\thanks{The author has received funding from the European Research Council under the European Community's Seventh Framework Programme (FP7/2007-2013 Grant Agreement no. 257039).}

\section*{Abstract}
We provide a description of the structure of $\aleph_0$-categorical trees and cycle-free partial orders.  First the maximal branches of $\aleph_0$-categorical tree are examined, followed by the configuration of the ramification orders, which are then combined to provided necessary and sufficient conditions for a tree to be $\aleph_0$-categorical in terms of these two things.  The classification of the $\aleph_0$-categorical cycle-free partial orders is found as a corollary.

\section{Introduction}

A countable structure is said to be $\aleph_0$-categorical (also known as countably categorical, or $\omega$-categorical) if it is the only countable model of its first order theory.  The Ryll-Nardzewski Theorem shows that this is equivalent to the automorphism group of the structure being oligomorphic.  This property is important in so many ways other than its intrinsic interest that it would be futile to give a good account of its importance here.

A tree, sometimes called a semi-linear order, is a partial order where initial sections are linear, but we can `split' as we go up the order.  A formal definition is given in Definition \ref{dfn:trees}.  A partial answer to the natural question `Which trees are $\aleph_{0}$-categorical?' can be found in Manfred Droste's memoir on transitive partially ordered sets \cite{Droste1985} which shows that there are no 4-transitive trees, no 3-homogeneous trees and classifies the 2- or 3- transitive and 2-homogeneous countable trees.  While all the 2 or 3-transitive trees are $\aleph_{0}$-categorical, this list is far from exhaustive.  For example, adding a root to any 2-transitive tree will result in something not 2-transitive, yet still $\aleph_{0}$-categorical.

Other studies of the automorphism groups of trees include Chicot's thesis, where she classifies the 1-transitive trees \cite{Chicot2004}, and the work of Droste, Holland and Macpherson, where the properties of these groups are studied in great detail \cite{DHM1989no1}, \cite{DHM1989no2}, \cite{DHM1989no3}.

An extremely elegant description of the $\aleph_{0}$-categorical linear orders was published by Joseph Rosenstein in 1969, which can be found in either \cite{Rosenstein1969} or \cite{Rosenstein1982}, and this result is instrumental for the work of this paper.

Since the proof of the main result of this paper adds colour predicates to the language of trees to ensure that some of the definable structure of a tree is respected by certain non-definable substructures, extending the result to the coloured trees is an easy corollary.  The $\aleph_0$-categorical coloured linear orders were classified by  Mwesigye and Truss in \cite{MwesigyeTruss2010}, who extended and made good use of one of Rosenstein's Propositions, namely Proposition 8.35 of \cite{Rosenstein1982}.  This lemma shows that in linear orders the $n$-orbits are determined by the 2-orbits of adjacent pairs by patching automorphisms together, and this method is at the heart of this paper.

Cycle-free partial orders (CFPOs) are a generalisation of trees, where you are allowed to branch as you move down the order, as well as up.  They were first proposed as objects of study by Rubin in \cite{Rubin91}.  The answer to various transitivity questions about CFPOs can be found in \cite{Warren1997}.

I would like to thank my Ph.D. supervisor, John Truss, for his extremely valuable advice and kind support.

\begin{dfn}\label{dfn:trees}
A tree is a partial order that satisfies the two additional axioms:
\begin{itemize}
\item $\forall x , y, z (x,y \leq z \rightarrow (x \leq y \:\mathrm{or}\: y \leq x))$
\item $\forall x, y \exists z (z \leq x,y)$
\end{itemize}
\end{dfn}

\begin{dfn}
A $\lambda$-coloured tree is a structure $\langle T, \leq, C_i \: : \: i \leq \lambda \rangle$ such that $\langle T, \leq \rangle$ is a tree, while the $C_i$ are mutually exclusive unary predicates.
\end{dfn}

\begin{definition}
If $x,y$ are elements of a partial order $T$ then the \textbf{meet} of $x$ and $y$ is written and defined as:
$$x \wedge y := \mathrm{sup} \lbrace t \in T \: : \: t \leq x,y \rbrace $$
\end{definition}

Note that $x \wedge y$ might not be an element of $T$, but in Section 2 we will show that there are always extensions of the tree in which these points exist.

\begin{definition}
A \textbf{cone} above a point $t$ is a maximal set $C$ such that
$$\forall c \in C \; t <c \quad \mathrm{and} \quad \forall c_0, c_1 \in C \; c_0 \wedge c_1 > t$$
Essentially the points strictly above $t$ should form a collection of trees, and a cone is one of the trees in this collection.

The \textbf{ramification order} of a point $t$ is the number of cones above $t$.
\end{definition}

\begin{definition}\label{def:irrational}
A tree $T$ is said to be \textbf{ramification complete} if it contains the meet of any two points, i.e. $x \wedge y \in T$ for every $x,y \in T$.  

The \textbf{ramification completion} of a tree $T$ is the intersection of all ramification complete trees $S$ such that $T \subseteq S$ and is written as $T^+$.

The elements of $T^+ \setminus T$ are called \textbf{irrational}.
\end{definition}

The ramification completion of a countable tree is always countable.

\begin{dfn}\label{definition:almosttransitive}
Let $T$ be a tree.  The \textbf{$n$-orbits} of $T$ are the following sets
$$\lbrace \phi(\bar{x}) \in T^n \: : \: \phi \in \aut(T) \rbrace$$
where $\bar{x} \in T^n$.  A tree is said to be \textbf{almost $n$-transitive} if it has only finitely many $n$-orbits.
\end{dfn}

\begin{theorem}
A tree $T$ is $\aleph_{0}$-categorical if and only if it is almost $n$-transitive for all $n$.
\end{theorem}
This is a reformulation of the Ryll-Nardzewski Theorem.  A proof for this context can be found in \cite{Hodges1993}, Theorem 7.3.1.

The next few lemmas and definitions allow us to reduce to the case $n=2$ when considering almost $n$-transitivity.

\begin{dfn}
The \textbf{completion of an $n$-tuple $p$} is a tuple of least length which contains $p$ and is closed under $\wedge$.

A \textbf{complete $n$-orbit} of $T$ is the orbit of some complete $n$-tuple.

$T$ is said to be \textbf{almost $n$-complete transitive} if it has finitely many complete $n$-orbits.
\end{dfn}

\begin{lemma}\label{lemma:complete}
A complete tree $T$ is almost $n$-complete transitive for every $n \geq 2$ if and only if $T$ is almost $m$-transitive for each $m \geq 2$.
\end{lemma}
\begin{proof}
If $T$ is almost $m$-transitive for each $m \geq 2$ then it is automatically almost $n$-complete transitive for every $n \geq 2$.  In the other direction, note that for each $n$-complete tuple there are only finitely many tuples whose completion is that tuple.
\end{proof}

The following theorem is a variant of Proposition 4.5 in Simon's paper, \cite{Simon2011}.

\begin{theorem}\label{Thm:2impliesn}
Let $T$ be a tree with $T=T^+$.  If $T$ is almost 2-transitive then $T$ is almost $n$-transitive for each $n \geq 2$.
\end{theorem}

\section{Linear Orders and Maximal Chains}

Since trees are built up from linear orders, this section will deal with the properties of linear orders and what kinds of linear orders can occur in an $\aleph_{0}$-categorical tree.

\begin{dfn}
If $\langle L_0 , <_0 \rangle$ and $\langle L_1 , <_1 \rangle$ are linear orders then their \textbf{concatenation}, denoted by $L_0 \,^{\wedge} L_1$ is the linear order $\langle L_0 \cup L_1 , < \rangle$, where
$$x< y \quad \mathrm{iff} \quad \left\lbrace
\begin{array}{r c c r}
(x,y \in L_0 & \mathrm{and} & x <_0 y) & \quad\mathrm{or} \\
(x,y \in L_1 & \mathrm{and} & x <_1 y) & \quad\mathrm{or} \\
(x \in L_0 & \mathrm{and} & y \in L_1)
\end{array} \right. $$
\end{dfn}

\begin{dfn}
$\langle \mathbb{Q}_n, <_{\mathbb{Q}_n}, C_1 \ldots C_n \rangle$ a countable dense linear order where the colours occur interdensely, i.e. for all $x$ and $y$ there are $z_1, \ldots z_n$ between $x$ and $y$ such that $C_i(z_i)$ holds for each $i$.
\end{dfn}

$\mathbb{Q}_n$ is the Fra\"iss\'e limit of $n$-coloured linear orders, and hence is $\aleph_0$-categorical.

\begin{dfn}
Let $\langle L_1 , <_1 \rangle, \ldots, \langle L_n , <_n \rangle$ be linear orders.  For each $q \in \mathbb{Q}_n$ we define $L(q)$ to be a copy of $\langle L_i , <_i \rangle$  	if $C_i(q)$.  The $\mathbb{Q}_n$-\textbf{shuffle} of  $$\langle L_1 , <_1 \rangle, \ldots, \langle L_n , <_n \rangle$$ denoted by $\mathbb{Q}_n(L_1, \ldots L_n)$, is the linear order $\langle \bigcup_{q \in \mathbb{Q}_n} L(q), < \rangle $, where
$$x< y \quad \mathrm{iff} \quad \left\lbrace
\begin{array}{c c c c}
x,y \in L(q) & \mathrm{and} & x <_i y & \mathrm{or}\\
x \in L(q) \, , \, y \in L(p) & \mathrm{and} & q <_{\mathbb{Q}_n} p
\end{array} \right. $$
\end{dfn}

\begin{theorem} [Rosenstein \cite{Rosenstein1969}, \cite{Rosenstein1982}] \label{theorem:rosenstein}
If $L$ is an $\aleph_{0}$-categorical linear order then $L$ is built up from singletons by a finite number of concatenations or shuffles.
\end{theorem}

This result was extended to the coloured linear orders by Mwesigye and Truss in the following theorem.

\begin{theorem}[Mwesigye, Truss \cite{MwesigyeTruss2010}]
A finite or countable coloured linear order $(A,\leq,C_0, \ldots)$ is $\aleph_{0}$-categorical if and only if it can be built up in finitely many steps from coloured singletons using concatenations or shuffles.
\end{theorem}

Rosenstein's theorem leads to a natural method of describing the countably categorical linear orders.

\begin{dfn}
A \textbf{term} is built as follows:

\begin{tabular}{r l}
\textbf{Singleton} & The singleton 1 is a term. \\
\textbf{Concatenation} & If $t_0, t_1$ are terms then $t_0 \, ^\wedge \, t_1$ is a term. \\
\textbf{$\mathbb{Q}_n$-shuffle} & If $t_0, \ldots t_{n-1}$ are terms then $\mathbb{Q}_n (t_0, \ldots, t_{n-1})$ is a term.
\end{tabular}

$\mathbb{Q}_n$-shuffle is allowed for all $n \in \mathbb{N}$.  A \textbf{finite term} is a term that represents a finite linear order.  Similarly, an \textbf{infinite term} is one that represents an infinite linear order.  \end{dfn}

The terms correspond to linear orders in the obvious way, and I will not be particularly careful about distinguishing the two.  That every $\aleph_{0}$-categorical linear order is represented by a term is Theorem \ref{theorem:rosenstein}, however it is possible for a linear order to have many different representations.

\begin{facts}\label{lemma:nfrep} \label{lemma:nfperm} \label{lemma:nfnest} \label{lemma:nfconc}
Let $t_0, \ldots, t_{n-1}$ be terms, let $m \leq n$ and let $f$ be a permutation of $n$.  We also let $\tau$ be either the empty set or one of the $t_i$. Then the following are isomorphic to $\mathbb{Q}_n(t_0, t_1, \ldots t_{n-1})$:
\begin{enumerate}
\item $\mathbb{Q}_n(t_{f(0)}, t_{f(1)}, \ldots t_{f(n-1)})$;
\item $\mathbb{Q}_{n+1} (t_0, \ldots, t_{n-1},t_m)$;
\item $\mathbb{Q}_{m+1}(t_0, \ldots t_{m-1}, \mathbb{Q}_n(t_0, \ldots t_{n-1}))$; and
\item $\mathbb{Q}_{m}(t_0, \ldots ,t_{m-1}) ^\wedge \tau ^\wedge \mathbb{Q}_{m}(t_0, \ldots ,t_{m-1})$.
\end{enumerate}
\end{facts}

Using this lemma it is possible to derive a unique representation of not only $\aleph_{0}$-categorical linear orders, but also $\aleph_{0}$-categorical coloured linear orders (by allowing coloured singletons to occur in our terms) and infinite concatenations of $\aleph_{0}$-categorical linear orders.  Such representations have certain properties that facilitate a proof regarding the maximal chains of trees. 

\begin{dfn}
We use induction over the formation of terms to define when a term is in \textbf{normal form} (n.f.).
\begin{enumerate}
\item All finite terms are in n.f..
\item A term of the form $\mathbb{Q}_{m}(t_0, \ldots ,t_{m-1})$ is in n.f. if:
\begin{enumerate}
\item all the $t_i$ are in n.f.; and
\item Number 2 and 3 of Facts \ref{lemma:nfrep} do not apply.
\end{enumerate}
If the $t_i$ are permuted then the term is unaltered.
\item A term of the form $t_0 \,^\wedge \ldots \,^\wedge t_{n-1}$ is in n.f. if all the $t_i$ are in n.f. and no $t_{i-1} \,^\wedge t_i \,^\wedge t_{i+1}$ or $t_{i} \,^\wedge \emptyset \,^\wedge t_{i+1}$ satisfy Number 4 of Facts \ref{lemma:nfconc}.
\end{enumerate}
A possibly infinite sequence of terms $(s_i)$ is said to be in \textbf{normal form} if:
\begin{enumerate}
\item each $s_i$ is in normal form;
\item no $s_{i-1} \,^\wedge s_i \,^\wedge s_{i+1}$ or $s_{i} \,^\wedge \emptyset \,^\wedge s_{i+1}$ satisfy Number 4 of Facts \ref{lemma:nfconc};
\item if $s_j$ is finite either:
\begin{enumerate}
\item $s_{j+1}$ is infinite; or
\item $(s_i)$ is an infinite sequence and $s_j = s_k = 1$ for all $k \geq j$.
\end{enumerate}
\end{enumerate}
\end{dfn}

The process of showing that such representations are unique and can describe every $\aleph_{0}$-categorical linear order is long and uninformative, so we shall not provide the proof, and simply state the pertinent facts about normal form representations:

\begin{facts} \label{facts:nf}
\textcolor{white}{gap}
\begin{enumerate}
\item For every sequence of terms $(t_i)$ there is a sequence in normal form $(t_i')$ that represents the same linear order as $(t_i)$.
\item If $(t_i)$ and $(s_i)$ are both in normal form and represent the same linear order then $t_i = s_i$ for every $i$.
\item If $(t_i)$ is in normal form then all contiguous subsequences of $(t_i)$ are also in normal form (excluding the case where $(t_i)$ ends in a tail of $1$ and the contiguous subsequence contains only a part of this tail).
\end{enumerate}
\end{facts}

These facts are required to show the following theorem about the possible maximal chains of an $\aleph_{0}$-categorical tree.

\begin{theorem}\label{theorem:maxchain}
If $T$ is an $\aleph_{0}$-categorical coloured tree then every maximal chain of $T$ is an $\aleph_{0}$-categorical coloured linear order.
\end{theorem}
\begin{proof}
Let $L$ be a maximal chain of $T$ such that $L$ is not $\aleph_{0}$-categorical as a linear order.  Every initial section of $L$ is an $\aleph_{0}$-categorical linear order.  Therefore $L$ is expressible as the concatenation of an infinite list of $\aleph_{0}$-categorical linear orders $(L_i)$.  We assume that $(L_i)$ is in normal form, which must be an infinite sequence as $L$ is not $\aleph_{0}$-categorical.

For each $i$ let $x_i \in L_i$.  The tree $T$ is $\aleph_{0}$-categorical so there must be an automorphism $\phi$ that sends $(x_0, x_{n+1})$ to $(x_0, x_{m+1})$ for some $m < n$.  The restriction of $\phi$ to $T^{\leq x_{n+1}}$ is an isomorphism from $L^{\leq x_{n+1}}$ to $L^{\leq x_{m+1}}$, so must send the set of predecessors of $x_{n+1}$ to the predecessors of $x_{m+1}$.

Suppose that $L_{n+1}$ is finite and therefore $\phi| _{T^{\leq x_{n+1}}}$ maps $L_0 \,^\wedge \ldots \,^\wedge L_n$ to $L_0 \,^\wedge \ldots \,^\wedge L_m$.  Thus the finite sequences $(L_i)^n_{i=0}$ and $(L_i)^m_{i=0}$ are isomorphic and Fact 3. of Facts \ref{facts:nf} shows that these sequences are in normal form.  Therefore $(L_i)^n_{i=0} = (L_i)^m_{i=0}$, and so $m=n$, giving a contradiction.

Suppose that $L_{n+1}$ is a shuffle, which we denote by $\mathbb{Q}_n(\tau_0, \ldots ,\tau_i)$.  We also suppose that $x_{n+1}$ is contained in a copy of $\tau_0$, and we use $z$ to label the point in $\mathbb{Q}_n$ that is replaced by that particular copy of $\tau_0$, and let $L_{n+1}'$ be the initial section of $L_{n+1}$ that corresponds to $(-\infty,z)$, the interval of $\mathbb{Q}_n$.

Since $(-\infty,z) \cong \mathbb{Q}_n$ we can deduce two isomorphisms, $L_{n+1}' \cong L_{n+1}$ and
$$L_0 \,^\wedge \ldots \, ^\wedge L_n \,^\wedge L_{n+1} \cong L_0 \,^\wedge \ldots \, ^\wedge L_n \,^\wedge L_{n+1}'$$
Thus the normal form representation of $L_{n+1}'$ is equal to the n.f. representation of $L_{n+1}$.  The function $\phi$ is an isomorphism, so the n.f. representation of $\phi(L_0 \,^\wedge \ldots \, ^\wedge L_n \,^\wedge L_{n+1}')$ is $L_0 \,^\wedge \ldots \, ^\wedge L_n \,^\wedge L_{n+1}$.

Therefore $\phi$ maps $L_i$ to itself for $i \leq n+1$, and thus the n.f. representation of $\tau_0^{\leq x_0}$ is also the n.f. representation of $L_{n+1} \,^\wedge \ldots \,^\wedge L_{m+1}^{\leq x_{m+1}}$.

$T$ is $\aleph_0$-categorical, so we may also assume that there is $m' \in \mathbb{N}$ such that there is an automorphism mapping $(x_0,x_{n+1})$ to $(x_0,x_{m'})$ and $m \not= m'$.  Again, we conclude that the n.f. representation of $\tau_0^{\leq x_0}$ is also the n.f. representation of $L_{n+1} \,^\wedge \ldots \,^\wedge L_{m'+1}^{\leq x_{m'+1}}$.

This is a contradiction, as the n.f representation of $\tau_0^{\leq x_0}$ is of fixed length.
\end{proof}

\begin{theorem}\label{Theorem:FinitelyManyChains}
If $T$ is a countable $\aleph_{0}$-categorical tree then $T$ has only finitely many maximal chains up to isomorphism.
\end{theorem}
\begin{proof}
Let $T$ be a tree with infinitely many non-isomorphic maximal chains, which we call $L_n$ for $n \in \mathcal{J}$.  For each $I \subseteq \mathcal{J}$ we introduce colour predicate $C_I$ such that $T \models C_I(a)$ if and only if
$$I = \lbrace i \in \omega \: : \: a \:\textnormal{is contained in a maximal chain isomorphic to} \: L_i \rbrace$$
We introduce the following notation:
$$\mathcal{I} := \lbrace I \subseteq \omega \: : \: T \models \exists x C_I(x) \rbrace$$

If $I \not=J$ and $T \models C_I(a) \wedge C_J(b)$ then WLOG there is a maximal chain $A$ such that $A$ passes through $a$, but no maximal chain passing through $b$ is isomorphic to $A$.  Any automorphism of $T$ that maps $a$ to $b$ will have to map $A$ to a maximal chain that contains $b$, showing that $a$ and $b$ lie in different orbits of $\aut(T)$, and hence
$$\aut(T) = \aut( \langle T, \leq, C_I \: : \: I \in \mathcal{I} \rangle )$$

If $\mathcal{I}$ is infinite then there are infinitely many 1-orbits of $T$, and $T$ cannot be $\aleph_0$-categorical, so we assume that $\mathcal{I}$ is finite.  Since $T$ has infinitely many maximal chains, $\bigcup \mathcal{I}$ is infinite, so there must be an infinite member of $\mathcal{I}$.

If $a < b$ and $T \models C_I(a) \wedge C_J(b)$ then $J \subsetneq I$, so if $I_0$ is a minimal element of $\mathcal{I}$ then there exists an $a_0 \in T$ such that $T^{\geq a_0}$ is mono-chromatically coloured by $C_{I_0}$.

Only finitely many of these $C_I$ can be realised but $T$ has infinitely many non-isomorphic maximal chains, so there is a $J \in \mathcal{I}$ such that $J$ is infinite.  Let $x \in T$ be such that $T \models C_J(x)$ and let $I_0, \ldots I_{k-1}$ be the minimal elements of $\mathcal{I}$ where there is a $y \geq x$ such that $T \models C_{I_j}(y)$.
$$J \subseteq \bigcup_{j<k} I_j$$
$J$ is infinite, so at least one of the $I_j$ is infinite.  We assume that $I_0$ is.  Let $y \in T$ realise $C_{I_0}$, and let $S:= T^{\geq y}$.  Since $I_0$ is minimal, $S$ is monochromatic.

In short, from our $T$ we have found another tree, $S$ where every element lies on a copy of two non-isomorphic linear orders.  Let $L_0$ and $L_1$ be these non-isomorphic maximal chains and let $\lbrace s_i \in S: i \in \omega \rbrace$ be an enumeration of a copy of $L_0$.  We build by induction an embedding of $L_0$ into $L_1$.

Since $s_0$ also lies on a copy of $L_1$
$$L_0^{\leq s_0} \cong T^{\leq s_0} \cong L_1^{\leq s_1}$$
so let $\phi_0$ be an isomorphism from $L_0^{\leq s_0}$ to $L_1^{\leq s_0}$.

Suppose we have defined $\phi_l$.  Let $\alpha_{l+1} \in \mathbb{N}$ be the least number such that $s_{\alpha_{l+1}} > s_{\alpha_l}$.  The element  $s_{\alpha_{l+1}}$ is contained in both an copy of $L_0$ and $L_1$, so once again
$$L_0^{\leq s_{\alpha_{l+1}}} \cong T^{\leq s_{\alpha_{l+1}}} \cong L_1^{\leq s_{\alpha_{l+1}}}$$
By the induction hypothesis, we also know that $L_0^{\leq s_{\alpha_{l}}} \cong T^{\leq s_{\alpha_{l}}} \cong L_1^{\leq s_{\alpha_{l}}}$ thus the intervals $(s_{\alpha_{l}},s_{\alpha_{l+1}}] \subseteq L_0$ and $(s_{\alpha_{l}},s_{\alpha_{l+1}}] \subseteq L_1$ are isomorphic.  Let this be witnessed by $\psi_{l+1}$, and we define $\phi_{l+1} := \phi_{l} \cup \psi_{l+1}$

Then $\bigcup \phi_{l}$ witnesses the fact that $L_0$ is isomorphic to an initial section of $L_1$.  By symmetry, $L_1$ is isomorphic to an initial section of $L_0$ as well.

Therefore, if $\tau_0$ and $\tau_1$ are the normal form representations of $L_0$ and $L_1$ respectively $\tau_1 = \tau_0 ^\wedge \sigma_0 \: \mathrm{and} \: \tau_0 = \tau_1 ^\wedge \sigma_1$ for terms $\sigma_0$ and $\sigma_1$.  Therefore $\tau_1 = \tau_1 ^\wedge \sigma_1^\wedge \sigma_0$ and so $\sigma_0 = \emptyset = \sigma_1$ and thus $\tau_0 = \tau_1$.  This shows that $L_0 \cong L_1$, giving a contradiction.
\end{proof}

\section{Trees}
\subsection{Ramification Predicates}

Trees contain more information than which linear orders occur as their maximal chains, so in order to classify the $\aleph_0$-categorical trees using them we need a way to encode that extra information.

\begin{dfn}
Let $T$ be an $\aleph_0$-categorical tree, and let $\lbrace L_k \: : \: k \leq l \rbrace$ be the maximal chains of $T$.  For each $L_k$, we enumerate the $1$-orbits.

We define $R^i_{(m,n)}$ to be a unary predicate that is only realised by $x \in T$ if $x$ lies on exactly $i$ maximal chains isomorphic to $L_m$ such that $x$ lies in the $n^{\mathrm{th}}$ orbit of $L_m$.  Additionally

$$J_T := \lbrace (i,(m,n)) \: : \: \exists x \: T \models R^i_{(m,n)}(x) \rbrace$$
\end{dfn}

\begin{lemma}\label{Lemma:NoKillingOrbits}
For all trees $T$
$$\aut(\langle T, \leq \rangle) \cong \aut(\langle T, \leq, R^i_{(m,n)} \, : \, (i,(m,n)) \in J_T \rangle)$$
\end{lemma}
\begin{proof}
If $\phi \in \aut(T)$ maps $y_0$ to $y_1$ also maps the maximal chains passing through $y_0$ to the maximal chains passing through $y_1$.  In particular, if $L$ is a maximal chain that contains $y_0$, then $\phi(L^{\geq y_0}) \cong \phi(L)^{\geq y_1}$.
\end{proof}

\begin{lemma}\label{Theorem:AddingStructuretoL}
If $\langle T, \leq \rangle$ is $\aleph_0$-categorical then for any maximal chain $L$ the structure $$\langle L, \leq, R^i_{(m,n)} \, : \, (i,(m,n)) \in J_T \rangle$$ is also $\aleph_0$-categorical.
\end{lemma}
\begin{proof}
We apply Theorem \ref{theorem:maxchain} to $\langle T, \leq, R^i_{(m,n)} \, : \, (i,(m,n)) \in J_T \rangle$.
\end{proof}

\subsection{Classification}\label{Section:CombineColours}

\begin{prop}
If $(T,<)$ is $\aleph_{0}$-categorical then $(T^+,<)$ is $\aleph_{0}$-categorical.
\end{prop}
\begin{proof}
The orbits of the irrational elements of $T^+$ are determined by the orbits of pairs from $T$.
\end{proof}

$T^+$ being $\aleph_{0}$-categorical is not enough to ensure that $T$ is $\aleph_{0}$-categorical.  This suggests that we need a way of restricting how points in $(T^+,<)$ can be deleted to ensure that the remaining structure is still $\aleph_{0}$-categorical.  Recall from Definition \ref{def:irrational} that an irrational point of $T^+$ is a point in $T^+ \setminus T$.

\begin{theorem}\label{thm:whereIdef}
Let $T$ be a tree and $I$ be a unary predicate for the irrational points.  Then $(T,<)$ is $\aleph_{0}$-categorical if and only if $(T^+,<,I)$ is $\aleph_{0}$-categorical. 
\end{theorem}
\begin{proof}
An automorphisms of $(T,<)$ extends uniquely to an automorphism of $(T^+,<,I)$, and automorphisms of $(T^+,<,I)$ restrict uniquely to an automorphism of $(T,<)$, so $\aut(T,<) \cong \aut(T^+,<,I)$.
\end{proof}

\begin{lemma}\label{Lemma:AddingIrrationalstoI}
If $(T^+,<,I)$ is $\aleph_{0}$-categorical then if $L$ is a maximal chain of $T^+$ then the linear order $(L,<,I)$ is $\aleph_{0}$-categorical.
\end{lemma}
\begin{proof}
The proof of Theorem \ref{theorem:maxchain} is easily adapted to this lemma.
\end{proof}

We are now ready to prove our main theorem about trees.

\begin{theorem}\label{Theorem:CCTreesMain}
$\langle T^+, I, <, R^i_{(m,n)} \, : \, (i,(m,n)) \in J_T \rangle$ is $\aleph_{0}$-categorical if and only if:
\begin{enumerate}
\item only finitely many of the $R^i_{(m,n)}$ are realised;
\item if $L$ is a maximal chain of $T^+$ then $\langle L, I, <, R^i_{(m,n)} \, : \, (i,(m,n)) \in J_T \rangle$ is $\aleph_{0}$-categorical; and
\item there are only finitely many maximal chains of $T^+$ up to isomorphism in the language $\langle  I, <, R^i_{(m,n)} \: : \: (i,(m,n)) \in J_T \rangle$.
\end{enumerate}
\end{theorem}

\pagebreak
\begin{proof}

$\Rightarrow$:  Since $\langle T^+, I, <, R^i_{(m,n)} \: : \:(i,(m,n)) \in J_T \rangle$ is $\aleph_{0}$-categorical it only has finitely many 2-orbits.  This means that only finitely many of the $R^i_m$'s can be realised.  Theorem \ref{theorem:maxchain} shows that $\langle L,<,I,R^i_{(m,n)} \: : \: (i,(m,n)) \in J_T \rangle$ is $\aleph_{0}$-categorical and Condition 3 is shown by Theorem \ref{Theorem:FinitelyManyChains}.

$\Leftarrow$:  Let $T$ be a tree that satisfies the conditions of the theorem, and suppose that $T$ has infinitely many 2-orbits.  Since only finitely many of the $R^i_{(m,n)}$ are realised, we may assume that there are $(x_0,y_0)$ and $(x_1,y_1)$ such that:
\begin{enumerate}
\item $(x_0,y_0)$ and $(x_1,y_1)$ belongs to a different 2-orbit;
\item $x_i < y_i$; and
\item $T \models R^i_{(m,n)}(y_0) \Leftrightarrow T \models R^i_{(m,n)}(y_1)$ for all $R^i_{(m,n)}$
\end{enumerate}
Since $y_0$ and $y_1$ satisfy the same $R^i_{(m,n)}$, they lie on maximal chains which are isomorphic to the same $\aleph_0$-categorical linear order, which we will call $L$.  We may assume that $\langle L, \leq , x_0, y_0 \rangle \cong \langle L, \leq , x_1, y_1 \rangle$.  We build an automorphism of $T$ that maps $(x_0, y_0)$ to $(x_1, y_1)$ inductively as follows:

\begin{description}
\item[Base Case]  Let $\phi_0 : \langle L, \leq , x_0, y_0 \rangle \rightarrow \langle L, \leq , x_1, y_1 \rangle$ be an isomorphism.  There are maximal chains $T_0$ and $S_0$ of $T$ that contain $(x_0, y_0)$ and $(x_1, y_1)$ respectively such that $T_0 \cong L \cong S_0$.  Thus $\phi_0$ can be viewed as a partial automorphism of $T$ that maps $T_0$ to $S_0$.
\item[Odd Step] Let $n$ be odd, let $T_n:= \mathrm{Dom}(\phi_n)$, let $S_n:= \mathrm{Im}(\phi_n)$ and let $t \in T_n \setminus T_{n-1}$.  For each cone of $t$ that is disjoint from $T_n$, pick a maximal chain.  We denote these maximal chains as $L_i(t)$, where $i \in I(t)$, an indexing set for each $t$.

Since $\phi_{n}$ is a partial automorphism of the language
$$\langle <, I, R^i_{(m,n)} \, : \, (i,(m,n)) \in J_T \rangle$$
the image $\phi_{n}(t)$ satisfies all of the same $R^i_{(m,n)}$ there is an isomorphism $\psi_{i,t}$ that maps $ L_i(t)$ to $K_i(\phi_n(t))$ and $\psi_{i,t}(t) =\phi_n(t)$.

This $\psi_{i,t}$ is also a partial automorphism. Since $\psi_{i,t}(t) = \phi_n(t)$ the union $\phi_n \cup (\psi_{i,t}|_{ L_i(t)^{> t}})$ is also a partial isomorphism.  Indeed, since each $L_i(t)$ lies in a different cone to any other $L_i(t)$,
$$\phi_{n+1} := \phi_n \cup \bigcup_{t \in T_n \setminus T_{n-1}} \bigcup_{i \in I(t)} \psi_{i,t}|_{ L_i(t)^{> t}}$$
is a partial isomorphism.

\item[Even Step] The even step is very similar to the odd step, except that we map maximal chains passing through the elements of $S_n \setminus S_{n-1}$ back, and expand $\phi_{n}$ by the inverses of these maps.

\end{description}

Then $\phi:= \bigcup_{n \in \mathbb{N}}\phi_n$ is an automorphism of $T$ that maps $(x_0,y_0)$ to $(x_1,y_1)$, showing that $T$ must be $\aleph_0$-categorical.
\end{proof}

This gives us necessary and sufficient conditions for $(T, \leq)$ to be $\aleph_{0}$-categorical.  A description of the coloured $\aleph_{0}$-categorical trees is contained in the proof of Theorem \ref{Theorem:CCTreesMain}, as we will now show.

\begin{cor}
A coloured tree $(T, < , C_0, \ldots, C_k)$ is $\aleph_{0}$-categorical iff
\begin{itemize}
\item only finitely many of the $R^i_{(m,n)}$ are realised;
\item $\langle L, <, I, C_0, \ldots, C_k,  R^i_{(m,n)} \: : \: (i,(m,n)) \in J_T \rangle$ is $\aleph_{0}$-categorical for every maximal chain $L$; and
\item there are only finitely many such maximal chains up to isomorphism in the language
$$\langle <, I, C_0, \ldots, C_k, R^i_{(m,n)} \: : \: (i,(m,n)) \in J_T  \rangle$$
\end{itemize}
where the $C_i$ are the colour predicates.
\end{cor}
\begin{proof}
The proof of Theorem \ref{Theorem:CCTreesMain} is easily adapted.
\end{proof}

\section{Cycle-free Partial Orders}

The aim of this section is to extend the above result to the cycle-free partial orders.  We shall give the definition of CFPO used in $\cite{Warren1997}$, after developing some notions analogous to those introduced at the start of this paper for trees.

\begin{dfn}
If $x,y$ are elements of a partial order $M$ then the \textbf{join} of $x$ and $y$ is written and defined as:
$$x \vee y := \mathrm{inf} \lbrace t \in T \: : \: t \leq x,y \rbrace $$ 
\end{dfn}

\begin{definition}
A partial order $M$ is said to be \textbf{path complete} if it for all $x,y \in M$:
\begin{enumerate}
\item if there exists a $z \in M$ such that $z \leq x,y$ then $x \wedge y \in T$; and
\item if there exists a $z \in M$ such that $z \geq x,y$ then $x \vee y \in T$.
\end{enumerate}

The \textbf{path completion} of a partial order $M$ is the intersection of all path complete partial orders $N$ such that $M \subseteq N$.  It is written as $M^+$.

The elements of $M^+ \setminus M$ are called \textbf{irrational}.
\end{definition}

\begin{dfn}[2.3.2 of \cite{Warren1997}]\label{dfn:connectingset}
If $M$ is a partial order and $a,b \in M$, then the n-tuple $C = \langle c_1,c_2, \ldots ,c_n \rangle$ (for $n \geq 2$) is said to be a \textbf{connecting set} from $a$ to $b$ in $M$, written $C \in C^M \langle a, b \rangle$, if the following hold:
\begin{enumerate}
\item $c_1=a, c_n=b, c_2, \ldots ,c_{n-1} \in M^+$;
\item if $j \not= i+1,i-1$ then $c_i \not\leq c_j$ and $c_j \not\leq c_i$; and
\item if $1<i<n$, then $c_{i-1} < c_i > c_{i+1}$ or $c_{i-1} > c_i < c_{i+1}$.
\end{enumerate}
\end{dfn}

\begin{dfn}[2.3.3 of \cite{Warren1997}]
Let $M$ be a partial order, $a,b \in M$, and let $C = \langle c_1,c_2,\ldots ,c_n \rangle$ be a connecting set from $a$ to $b$ in $M$.  Let $\sigma_k$ (for $ 1< k < n$) be maximal chains in $M^+$ with endpoints $c_k,c_{k+1} \in \sigma_k$, such that if $x \in \sigma_i \cap \sigma_j$ for some $i < j$, then $j = i+1$ and $x = c_{i+1}$.  Then we say that $P = \bigcup_{0<k<n} \sigma_k$ is a \textbf{path} from $a$ to $b$ in $M$.
\end{dfn}

\begin{dfn}
A partial order $M$ is said to be a \textbf{cycle-free partial order} (CFPO) if for all $x,y \in M$ there is at most one path between $x$ and $y$ in $M^+$.  If it exists, this unique path is denoted by $\path{x,y}$.
\end{dfn}

\begin{dfn}
$\alt{}$ is the partial order with the domain $ \lbrace a_i \; : \; i \in \mathbb{Z} \rbrace $ ordered by:
\begin{itemize}
\item if $i$ is odd then $ a_{i-1} > a_{i} < a_{i+1} $; and
\item if $i$ is even then $ a_{i-1} < a_{i} > a_{i+1} $.
\end{itemize}
$\alt{n}$ is defined to be $\alt{}$ restricted to $\lbrace a_0, \ldots a_{n-1} \rbrace$.  Note that flipping the order does not affect the definition of $\alt{}$, but does affect $\alt{n}$.  We will write $\alt{n}^*$ for the reverse ordering of $\alt{n}$.
\end{dfn}

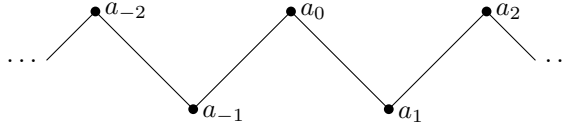
\begin{figure}[!ht]
\begin{center}
\begin{tikzpicture}[scale=0.13]
\draw (-25,0) -- (-20,5) -- (-10,-5) -- (0,5) -- (10,-5) -- (20,5) -- (25,0);

\fill (-20,5) circle (0.5);
\fill (-10,-5) circle (0.5);
\fill (0,5) circle (0.5);
\fill (20,5) circle (0.5);
\fill (10,-5) circle (0.5);

\draw[anchor=west] (-20,5) node {$a_{-2}$};
\draw[anchor=west] (-10,-5.5) node {$a_{-1}$};
\draw[anchor=west] (0,5) node {$a_{0}$};
\draw[anchor=west] (10,-5.5) node {$a_{1}$};
\draw[anchor=west] (20,5) node {$a_{2}$};

\draw[anchor=east] (-25,0) node {$\ldots$};
\draw[anchor=west] (25,0) node {$\ldots$};

\end{tikzpicture}
\end{center}
\caption{The Alternating Chain}
\end{figure}

That $\alt{}$ is a CFPO is readily apparent.

\begin{prop}\label{prop:altnotCC}
Let $M$ be a CFPO.  If $\alt{n} \subseteq M$ for all $n \in \mathbb{N}$ then $M$ is not $\aleph_0$-categorical.
\end{prop}
\begin{proof}
Paths are preserved by automorphisms, so pairs joined by different length paths must lie in different 2-orbits.
\end{proof}

\begin{dfn}
We say that $M$, a CFPO, is a $\cfpo{n}$ if $M$ embeds $\alt{n}$ or $\alt{n}^*$, but not $\alt{n+1}$ or $\alt{n+1}^*$.
\end{dfn}

We may therefore restrict our attention to the $\cfpo{n}$s.  However, it was shown in \cite{BarhamTreelike} that if $M$ is a $\cfpo{n}$ then there is a coloured tree $T(M)$ such that $\aut(M) \cong \aut(T(M))$, where this isomorphism is as permutations groups (i.e. the orbits of $M$ and $T(M)$ are equal), so we conclude the following:

\begin{cor}
Let $M$ be a $\cfpo{n}$.  $M$ is $\aleph_0$-categorical if and only if $T(M)$ is as well.
\end{cor}

\bibliography{bib.bib}{}
\bibliographystyle{plain}

\end{document}